\documentclass[oneside,english]{amsart}

\usepackage{amssymb}
\usepackage{amscd}

\numberwithin{equation}{section} 
\numberwithin{figure}{section} 
\theoremstyle{plain}
\newtheorem*{thm*}{Theorem}
\theoremstyle{plain}
\newtheorem{thm}{Theorem}[section]
\theoremstyle{definition}
\newtheorem{defn}[thm]{Definition}
\theoremstyle{plain}
\newtheorem{lem}[thm]{Lemma}
\theoremstyle{plain}

\theoremstyle{plain}

\theoremstyle{remark}

\theoremstyle{remark}
\newtheorem*{acknowledgement*}{Acknowledgement}

\begin{document}

\title[Invariant volume forms and first integrals]{Invariant volume forms and first integrals for geodesically equivalent Finsler metrics}

\author[Bucataru]{Ioan Bucataru}
\address{Faculty of  Mathematics \\ Alexandru Ioan Cuza University \\ Ia\c si, 
  Romania}
\email{bucataru@uaic.ro}
\urladdr{http://www.math.uaic.ro/\textasciitilde{}bucataru/}

\date{\today}

\begin{abstract}
Two geodesically (projectively) equivalent Finsler metrics determine a set of invariant volume forms on the projective sphere bundle. Their proportionality factors are geodesically invariant functions and hence they are first integrals. Being $0$-homogeneous functions, the first integrals are common for the entire projective class. In Theorem 1.1 we provide a practical and easy way of computing these first integrals as the coefficients of a characteristic polynomial.
\end{abstract}



\maketitle

\section{Introduction}

Finsler geometry has been characterised by Chern as being ``just Riemannian geometry without the quadratic restriction'' \cite{Chern96}. However, many important results from Riemannian geometry cannot be easily extended to the Finslerian framework without using techniques and tools that are specific to Finsler geometry, \cite{BC20}.

In this work we will extend the results of Matveev and Topalov from
\cite{MT98}, regarding the existence of first integrals for
geodesically equivalent metrics, from the Riemannian to the
Finslerian settings.  

It has been shown by Matveev and Topalov that two geodesically
equivalent Riemannian metrics, on an $n$-dimensional manifold,
determine a set of $n$ first integrals, \cite[Theorem 1]{MT98}.  In this work, we propose an extension of this result to the Finslerian context. For two projectively related Finsler metrics, their Hilbert $2$-forms are geodesically invariant. We use these Hilbert forms to construct a set of $n$ volume forms on the projective sphere bundle $SM$, which are invariant by the unitary geodesic vector field (the Reeb vector field for an induced contact structure on $SM$). Therefore the proportionality factors of these volume forms provide $n-1$ geodesically invariant functions. These $n-1$ functions are $0$-homogeneous in the fibre coordinates and hence they are common first integrals for all Finsler metrics in the same projective class. The missing, nth first integral, is the energy function, which is not $0$-homogeneous and cannot be obtained using this technique. 

We consider $M$ a smooth manifold, of dimension $n\geq 2$ and $TM$ its
tangent bundle. A continuous, positive, $1$-homogeneous  (in the fiber
coordinates) function $F:TM\to [0, +\infty)$, and smooth on $T_0M=TM\setminus \{0\}$, defines a Finsler structure if the metric tensor
$$
g_{ij}(x,y)=\frac{1}{2}\frac{\partial^2F^2}{\partial y^i \partial y^j}(x,y) 
$$
is non-degenerate on $T_0M$. Here $(x,y)\in TM$, with $x\in M$ and
$y\in T_xM$.

In \cite[Theorem 1]{MT98}, Matveev and Topalov use the
characteristic polynomial of a $(1,1)$-type tensor, constructed with
two geodesically equivalent Riemannian metrics, to generate $n$ first integrals that are quadratic in velocities. Two alternative ways for obtaining quadratic first integrals for projectively equivalent Riemannian metrics have been proposed by Crampin in \cite{Crampin03}. In our work we will use the angular metric to construct a characteristic polynomial, \eqref{cpq}, whose coefficients are first integrals. 

The metric tensor $g_{ij}$ of the Finsler structure $F$ can be expressed in terms of the angular metric $h_{ij}$ as follows:
\begin{eqnarray}
g_{ij}=h_{ij}+\frac{\partial F}{\partial y^i} \frac{\partial F}{\partial
  y^j}=h_{ij}+F_{y^i}F_{y^j}, \quad h_{ij}=F\frac{\partial^2
  F}{\partial y^i\partial y^j}=FF_{y^iy^j}. \label{gh}
\end{eqnarray}
The metric tensor $g_{ij}$ has rank $n$ if and only if the angular
metric $h_{ij}$ has rank $n-1$, \cite[Proposition
16.2]{Matsumoto86}. The angular metric plays an important role in
projective Finsler geometry, \cite{BM12}. 

Consider $F$ and $\widetilde{F}$ two Finsler metrics on the same
manifold $M$ of dimension $n\geq 2$. The Finsler metrics $F$ and $\widetilde{F}$ are geodesically equivalent, or projectively related, if they have the same geodesics (considered as oriented unparameterised curves). 

For two projectively related Finsler metrics $F$ and
$\widetilde{F}$ consider the characteristic polynomial:
\begin{eqnarray}
Q(\Lambda)  = \det \left( \mathcal{H}^i_j + \Lambda \delta^i_j
  \right) = \sum_{\alpha=1}^n f_{\alpha}\Lambda^{\alpha}. \label{cpq}
\end{eqnarray}
The $(1,1)$-type tensor, of rank $n-1$, $\mathcal{H}^i_j$ is determined by the angular metric $\widetilde{h}_{kj}$ of the Finsler structure $\widetilde{F}$, the contravariant metric tensor $g^{ik}$ of the Finsler structure $F$ and the two Finsler metrics $F$ and $\widetilde{F}$:
\begin{eqnarray}
\mathcal{H}^i_j =\dfrac{F}{\widetilde{F}} g^{ik}\widetilde{h}_{kj} = \dfrac{F}{\widetilde{F}} g^{ik}\left(
  \widetilde{g}_{kj} - \frac{\partial \widetilde{F}}{\partial y^k}
  \frac{\partial \widetilde{F}}{\partial y^j} \right) = \dfrac{F}{\widetilde{F}} g^{ik} \left( \widetilde{g}_{kj} - \dfrac{1}{\widetilde{F}^2} \widetilde{g}_{ks}y^k \widetilde{g}_{jl}y^l \right). \label{Hij}
\end{eqnarray}

We formulate now the main theorem of our work.

\begin{thm} \label{mainthm}
  Consider $F$ and $\widetilde{F}$ two projectively related Finsler
  metrics. The coefficients $f_{\alpha}$, $\alpha\in \{1,...,n-1\}$, of the characteristic polynomial \eqref{cpq}, 
  are first integrals for the geodesic spray $S$ of the Finsler metric $F$, which means that $S(f_\alpha)=0$.
\end{thm}
For $\alpha=n$ we have that $f_{n}=1$. Therefore, the coefficients of the polynomial \eqref{cpq} give $n-1$ non-trivial first integrals. In order to prove this, we will show first that the $0$-homogeneous functions $f_{\alpha}$ represent proportionality factors for some  volume forms on the projective sphere bundle $SM$. The key point for the proof of Theorem \ref{mainthm} is that these volume forms are invariant by the unitary geodesic spray $S/F$, which is the Reeb vector field of the contact manifold $(SM, F_{y^i}dx^i)$. The idea  of using invariant differential $2$-forms to obtain invariant functions has been used before by Tabachnikov in \cite{T99}. A modern reformulation of Tabachnikov's results has been used recently by \'Alvarez-Paiva to prove some  rigidity results for geodesically equivalent Finsler metrics, \cite[Theorem VI]{AP21}.

For $\alpha=1$, the first integral $f_1$ can be expressed as:
\begin{eqnarray}
f_1=\frac{F^{n+1}}{\widetilde{F}^{n+1}}\frac{\det \widetilde{g}}{\det
  g}. \label{io}
  \end{eqnarray}
An equivalent expression for the first integral $f_1$ has been obtained recently in \cite[Lemma 3.1]{BCC21} using different techniques. In formula \eqref{ri0}, we present an equivalent expression of this first integral, which in the Riemannian case it corresponds to the Painlev\'e first integral $I_0$, \cite[Remark 1]{TM03}.

Another first integral can be obtained for $\alpha=n-1$:
\begin{eqnarray}
f_{n-1}=\operatorname{Tr}\left(\mathcal{H}^i_j\right) = \dfrac{F}{\widetilde{F}}g^{ij} \widetilde{h}_{ij} =  \dfrac{F}{\widetilde{F}^3}g^{ij}\left(\widetilde{g}_{ij} \widetilde{F}^2 - \widetilde{g}_{ik}y^k \widetilde{g}_{jl}y^l \right).  \label{in1} 
\end{eqnarray} 
An equivalent expression of this first integral is given by formula \eqref{ri1}, which in the Riemannian case corresponds to the first integral, $I_1$, obtained by Topalov and Matveev, \cite[Theorem 1]{TM03}.

In dimension $2$, Foulon and Ruggiero have shown the existence of a first integral for the geodesic flow of a $k$-basic Finsler surface, \cite{FR16}. This result was extended to arbitrary dimension in \cite{BCC21}, by providing a class of Finsler manifolds that admit a first integral. 

A different approach for obtaining first integrals in Finsler geometry is due to Sarlet, who provides in \cite{Sarlet07} a recursive scheme of first integrals of the geodesic flow of a Finsler manifold. Using Sarlet's approach, for two projectively related Finsler metrics, the tensor 
\begin{eqnarray*}
K^i_j=\left(\dfrac{\det \widetilde{g}}{\det g}\right)^{\frac{1}{n+1}}\widetilde{g}^{ik}g_{kj}   
\end{eqnarray*}
satisfies \cite[(3)]{Sarlet07}, which means that it is a special conformal Killing tensor. The first integral generated by \cite[Theorem 3]{Sarlet07}, using this tensor $K$, is an equivalent expression of the first integral $f_1$, see formula \eqref{ri0}.

\section{Projectively equivalent Finsler metrics}

Consider $M$ a smooth manifold, of dimension $n\geq 2$, $TM$ its tangent bundle and $T_0M=TM\setminus\{0\}$ the tangent bundle with the zero section removed. We denote the local coordinates on the base manifold $M$ by $(x^i)$ and by $(x^i, y^i)$ the induced canonical
coordinates on $TM$ (and $T_0M$).

The tangent bundle $TM$ carries two canonical structures: the Liouville (dilation) vector field, ${\mathcal  C}=y^i\frac{\partial}{\partial y^i}$, and the tangent structure  (vertical endomorphism),
$J=\frac{\partial}{\partial y^i}\otimes dx^i$.   

A Finsler structure is defined by a continuous function $F:TM\to [0,
+\infty)$, smooth on $T_0M$, which satisfies the following
two assumptions:
\begin{itemize}
\item[$F_1$)] $F$ is positively $1$-homogeneous ($1^+$-homogeneous): $F(x,\lambda y)=\lambda F(x,y)$, $\forall
    \lambda >0, \forall (x,y)\in T_0M$;
  \item[$F_2$)] the Hessian of $F^2$
    $$g_{ij}(x,y)=\frac{1}{2}\frac{\partial^2F^2}{\partial y^i
      \partial y^j}(x,y) \ \textrm{is \ regular}. $$ 
  \end{itemize}
A pair $(M,F)$ is called a Finsler manifold. On a Finsler manifold we
identify the projective sphere bundle $SM=T_0M/{\mathbb{R}_+}$ with the indicatrix bundle
$IM=\{(x,y)\in TM, F(x,y)=1\}$. We note that functions defined on
$T_0M$ that are invariant under positive rescaling (are
$0^+$-homogeneous) can be restricted to $SM$. Some other geometric structures on $T_0M$ can also be restricted to $SM$. For example, a $0^+$-homogeneous form $\omega \in \Lambda^k(T_0M)$ can be restricted to $SM$ if and only if $i_{\mathcal{C}}\omega=0$.

In this work we will use the Fr\"olicher-Nijenhuis formalism to describe the geometric setting on a Finsler manifold, \cite[Chapter 2]{GM00}. For a vector valued $k$-form $K$, we will denote by $d_K$ the corresponding Lie derivation of degree $k$. When $K$ is a vector field, the corresponding derivation of degree $0$ is the standard Lie derivation $\mathcal{L}_K$.

For a Finsler metric $F$, its Hilbert $1$ and $2$-forms: 
\begin{eqnarray*}
d_JF=\frac{\partial F}{\partial y^i}dx^i, \quad
  dd_JF=\frac{1}{2}\left(\frac{\partial^2 F}{\partial y^i\partial
  x^j} - \frac{\partial^2 F}{\partial y^j\partial x^i}\right)
  dx^i\wedge dx^j +
  \frac{1}{F}h_{ij} dy^i\wedge dx^j  \label{hilbert12}
\end{eqnarray*}
are $0^+$-homogeneous. Moreover, $i_{\mathcal{C}}d_JF=0$ and $i_{ \mathcal{C}}dd_JF=0$ and therefore we can restrict these Hilbert forms to $SM$, $d_JF\in \Lambda^1(SM)$, $dd_JF\in \Lambda^2(SM)$.

The regularity condition $F_2$, for the metric tensor of a Finsler structure, ensures that there is a unique vector field $S\in \mathfrak{X}(T_0M)$ that satisfies the Euler-Lagrange equations:
\begin{eqnarray*}
{\mathcal L}_Sd_JF^2 - dF^2=0. \label{dsf2}
\end{eqnarray*}
The vector field $S$ is $2^+$-homogeneous and it is called the
geodesic vector field (spray) of the Finsler metric $F$. In local coordinates, the geodesic spray is given by:
\begin{eqnarray*}
S=y^i\frac{\partial}{\partial x^i} - 2G^i \frac{\partial}{\partial y^i},
\end{eqnarray*}
where $G^i$ are $2^+$-homogeneous functions, locally defined, on $T_0M$. 

For a Finsler metric $F$, we use the geometric framework induced by its geodesic spray $S$. We consider the horizontal and the vertical distributions on $T_0M$ determined by the horizontal and vertical projectors, \cite{Grifone72}: 
\begin{eqnarray*}
h=\frac{1}{2}\left(\operatorname{Id} - [S, J]\right), \quad v=\frac{1}{2}\left(\operatorname{Id} + [S, J]\right). \label{hv}
\end{eqnarray*}   
Locally, the two projectors $h$ and $v$ have the following expressions:
\begin{eqnarray*}
h=\frac{\delta}{\delta x^i}\otimes dx^i, \quad v=\frac{\partial}{\partial y^i}\otimes \delta y^i, \quad \textrm{ where: \  } \frac{\delta}{\delta x^i} = \frac{\partial}{\partial x^i} - \frac{\partial G^j}{\partial y^i}\frac{\partial}{\partial y^j}, \quad \delta y^i = dy^i + \frac{\partial G^i}{\partial y^j}dx^j.
\end{eqnarray*} 
\begin{defn}
Two Finsler metrics $F$ and $\widetilde{F}$ are geodesically
equivalent (projectively related) if their geodesics coincide as
unparameterised oriented curves. 
\end{defn}
Two Finsler structures $F$ and $\widetilde{F}$ are projectively
related if and only if either one of the following equivalent
Rapcs\'ak equations are satisfied:
\begin{itemize}
\item[($R_1$)] ${\mathcal L}_Sd_J\widetilde{F}=d\widetilde{F}$;
\item[($R_2$)] $i_Sdd_J\widetilde{F}=0$;
\item[($R_3$)] ${\mathcal L}_Sdd_J\widetilde{F}=0$;
\item[($R_4$)] $d_hd_J\widetilde{F}=0$. 
  \end{itemize}
For an extensive set of Rapcs\'ak equations, we refer to  \cite[Theorem 9.2.22]{SLK14}.
The third Rapcs\'ak equation $(R_3)$ ensures that the Hilbert $2$-form
$dd_J\widetilde{F}$ (of the Finsler metric $\widetilde{F}$) is
geodesically invariant (with respect to the Finsler metric $F$).

The fourth Rapcs\'ak equation $(R_4)$ ensures that the Hilbert $2$-forms
$dd_JF$ and $dd_J\widetilde{F}$ have a special form with respect to the horizontal and vertical distributions induced by the Finsler metric $F$:
\begin{eqnarray}
dd_JF=\frac{1}{F}h_{ij} \delta y^i \wedge dx^j, \quad dd_J\widetilde{F}=\frac{1}{\widetilde{F}}\widetilde{h}_{ij} \delta y^i \wedge dx^j. \label{ddjf}
\end{eqnarray}
According to this formula, the Hilbert $2$-forms $dd_JF$ and $dd_J\widetilde{F}$ vanish whenever their arguments are either both horizontal vector fields or both vertical vector fields.  

\section{Invariant volume forms on the projective sphere bundle $SM$}

We have seen already that for a Finsler metric $F$, its Hilbert $1$-form $d_JF$ is $0^+$-homogeneous and hence it can be restricted to the projective sphere bundle $SM$. This $1$-form is a contact structure on the $(2n-1)$-dimensional manifold $SM$, which means that $d_JF \wedge
\left(dd_JF\right)^{(n-1)} \neq 0$. Therefore, the contact manifold $(SM, d_JF)$ has a canonical volume form:    
\begin{eqnarray}
\Omega_{F}=\frac{(-1)^{(n-1)(n-2)/2}}{(n-1)!} d_JF\wedge \left(dd_JF\right)^{(n-1)} \in \Lambda^{2n-1}(SM). \label{osm}
\end{eqnarray}  
For an introduction to contact structures and induced volume forms we refer to \cite[\S 10.1]{KLR07}. For the contact manifold $(SM, d_JF)$ one can easily see that the corresponding Reeb vector field is the normalized geodesic spray $S/F$, since it satisfies: 
\begin{eqnarray*}
d_JF\left(\dfrac{S}{F}\right)=1, \ i_{S/F}dd_JF=\dfrac{1}{F}i_Sdd_JF=\dfrac{1}{F}\left(\mathcal{L}_Sd_JF - di_Sd_JF\right) =0.
\end{eqnarray*}

Consider now $F$ and $\widetilde{F}$ two projectively related Finsler metrics and their Hilbert forms: $d_JF\in \Lambda^1(SM)$, $dd_JF\in \Lambda^2(SM)$ and $dd_J\widetilde{F} \in \Lambda^2(SM)$. 
Using these ingredients, we can define the following $(2n-1)-$forms, for each $\alpha\in \{1,...,n\}$:
\begin{eqnarray}
\Omega_\alpha=\frac{(-1)^{(n-1)(n-2)/2}}{(\alpha-1)! (n-\alpha)!}d_JF\wedge
  \left(dd_J F\right)^{(\alpha-1)} \wedge
  \left(dd_J \widetilde{F}\right)^{(n-\alpha)} \in \Lambda^{2n-1}(SM). \label{omegak} 
\end{eqnarray}
In the next section, we will prove that the volume forms \eqref{omegak} are invariant by the Reeb vector field $S/F$. 

Now, we express the proportionality factors between $\Omega_\alpha$ and $\Omega_{F}=\Omega_n$ using two characteristic polynomials, which we construct with the help of the angular metrics $h_{ij}$ and $\widetilde{h}_{ij}$ of the two Finsler metrics $F$ and $\widetilde{F}$:   
\begin{eqnarray*}
P(\Lambda) & = & \det \left( \dfrac{F}{\widetilde{F}}\widetilde{h}_{ij} + \Lambda g_{ij}  \right) = \sum_{\alpha=1}^{n}
  \delta_{\alpha} \Lambda^{\alpha}, \label{pl} \\
  Q(\Lambda) & = &  \det \left( \mathcal{H}^i_j + \Lambda \delta^i_j
  \right) =\frac{1}{\det g}  P\left(\Lambda\right) = \sum_{\alpha=1}^n f_{\alpha}\Lambda^{\alpha}.
\end{eqnarray*}
The $(1,1)$-type tensor, of rank $n-1$, $\mathcal{H}^i_j$ is
given by  
\begin{eqnarray*}
\mathcal{H}^i_j =\dfrac{F}{\widetilde{F}} g^{ik}\widetilde{h}_{kj} = \dfrac{F}{\widetilde{F}} g^{ik}\left(
  \widetilde{g}_{kj} - \frac{\partial \widetilde{F}}{\partial y^k}
  \frac{\partial \widetilde{F}}{\partial y^j} \right).
\end{eqnarray*}
We note that the coefficients $\delta_{\alpha}$ of the polynomial $P$ are not globally defined functions on $T_0M$ (or $SM$), under a change of coordinates they obey the same transformation law as $\det g$. However, the quotient
$f_{\alpha} = \delta_{\alpha}/\det g$, the coefficients of the polynomial $Q$, are globally defined functions on $T_0M$, and being $0^+$-homogeneous functions, they are globally defined on $SM$, for each $\alpha\in \{1,2,...,n\}$. 

The two polynomials $P$ and $Q$ have no free terms, $P(0)=\det(\frac{F}{\widetilde{F}}\widetilde{h}_{ij})=0$ and $Q(0)=\det(\mathcal{H}^i_j)=0$. In the next lemma, we provide an explicit formula for the coefficients $\delta_{\alpha}$ of the polynomial $P$.

\begin{lem} \label{lemdeltak} The coefficients $\delta_{\alpha}$ of the polynomial $P$ are $0^+$-homogeneous functions, given by
\begin{eqnarray} \label{deltak}
\delta_{\alpha} & = & \dfrac{F^{n-\alpha}}{\widetilde{F}^{n-\alpha}} \frac{1}{(\alpha-1)! (n-\alpha)!}\sum_{\sigma_1, \sigma_2 \in S_n} \varepsilon(\sigma_1\sigma_2)  h_{\sigma_1(1)\sigma_2(1)}
               \cdots  h_{\sigma_1(\alpha-1) \sigma_2(\alpha-1)} \\
  &  & \cdot  
                 \widetilde{h}_{\sigma_1(\alpha)\sigma_2(\alpha)} \cdots
                 \widetilde{h}_{\sigma_1(n-1)\sigma_2(n-1)} \frac{\partial
            F}{\partial y^{\sigma_1(n)}} \frac{\partial F}{\partial
            y^{\sigma_2(n)}}. \nonumber
  \end{eqnarray}
\end{lem}
\begin{proof}
From the definition of the polynomial $P$ we have 
\begin{eqnarray*}
P(\Lambda)=\det \left( \dfrac{F}{\widetilde{F}} \widetilde{h}_{ij} +  \Lambda g_{ij} \right) = \sum_{\sigma\in S_n} \varepsilon(\sigma)\left( \dfrac{F}{\widetilde{F}} \widetilde{h}_{1\sigma(1)} + \Lambda g_{1\sigma(1)}\right) \cdots \left(\dfrac{F}{\widetilde{F}}  \widetilde{h}_{n\sigma(n)} + \Lambda g_{n\sigma(n)}\right) = \sum_{\alpha=1}^{n} \delta_{\alpha} \Lambda^{\alpha}. \label{pl1}
\end{eqnarray*}
Therefore, the coefficient $\delta_{\alpha}$ collects all terms from the above sum that contain $\alpha$ factors of elements of $g_{ij}$ and $(n-\alpha)$ factors of elements of $\frac{F}{\widetilde{F}}\widetilde{h}_{ij}$:
\begin{eqnarray*}
\delta_{\alpha} & \stackrel{(a)}{=} & \dfrac{F^{n-\alpha}}{\widetilde{F}^{n-\alpha}} \sum_{\sigma \in S_n} \varepsilon(\sigma) \sum_{\substack{1\leq i_1<\cdots < i_{\alpha}\leq n \\ 1\leq j_1<\cdots < j_{n-\alpha}\leq n \\ \{i_1,...,i_{\alpha}\}\cap \{j_1,..., j_{n-\alpha}\}=\emptyset}} g_{i_1\sigma(i_1)} \cdots g_{i_{\alpha}
\sigma(i_{\alpha})} \widetilde{h}_{j_1\sigma(j_1)} \cdots \widetilde{h}_{j_{n-{\alpha}}\sigma(j_{n-{\alpha}})} \\
& \stackrel{(b)}{=} & \dfrac{F^{n-\alpha}}{\widetilde{F}^{n-\alpha}} \frac{1}{\alpha!(n-\alpha)!} \sum_{\sigma, \sigma_1 \in S_n} \varepsilon(\sigma) g_{\sigma_1(1)\sigma(\sigma_1(1))} \cdots g_{\sigma_1(\alpha)\sigma(\sigma_1(\alpha))} \widetilde{h}_{\sigma_1(\alpha+1)\sigma(\sigma_1(\alpha+1))} \cdots \widetilde{h}_{\sigma_1(n)\sigma(\sigma_1(n))} \\
 & \stackrel{(c)}{=} & \dfrac{F^{n-\alpha}}{\widetilde{F}^{n-\alpha}} \frac{1}{\alpha!(n-\alpha)!} \sum_{\sigma_1, \sigma_2 \in S_n} \varepsilon(\sigma_1 \sigma_2) g_{\sigma_1(1)\sigma_2(1)} \cdots g_{\sigma_1(\alpha)\sigma_2(\alpha)} \widetilde{h}_{\sigma_1(\alpha+1)\sigma_2(\alpha+1)} \cdots \widetilde{h}_{\sigma_1(n)\sigma_2(n)} \\
 & \stackrel{(d)}{=} & \dfrac{F^{n-\alpha}}{\widetilde{F}^{n-\alpha}} \frac{1}{(\alpha-1)!(n-\alpha)!} \sum_{\sigma_1, \sigma_2 \in S_n} \varepsilon(\sigma_1 \sigma_2) h_{\sigma_1(1)\sigma_2(1)} \cdots h_{\sigma_1(\alpha-1)\sigma_2(\alpha-1)} \\ & & \cdot \widetilde{h}_{\sigma_1(\alpha)\sigma_2(\alpha)} \cdots \widetilde{h}_{\sigma_1(n-1)\sigma_2(n-1)}  \frac{\partial F}{\partial y^{\sigma_1(n)}} \frac{\partial F}{\partial y^{\sigma_2(n)}},
\end{eqnarray*}
 which gives formula \eqref{deltak}. 
 
The equality (a) can be obtained also using Marvin Marcus' formula, \cite[(1)]{Marcus90}, for the sum of a determinant:
 \begin{eqnarray}
 \det(A+B)=\sum_{r} \sum_{a,b} (-1)^{s(a)}(-1)^{s(b)}\det\left(A[a\vert b]\right) \det\left(B[a\vert b]\right), \label{detab}
 \end{eqnarray}
 where $r$ is an integer from $0$ to $n$, and $a$ an $b$ are increasing integer sequences of length $r$ chosen from $1$ to $n$. Moreover, $A[a\vert b]$  is the $r$-square matrix obtained by taking rows $a$ and columns $b$, while $B[a\vert b]$ is the $(n-r)$-square matrix taking the complementary rows to $a$ and $b$, respectively. Also, $s(a)$ and $s(b)$ are the sums of the integers in $a$ and $b$, respectively.  

We can apply formula \eqref{detab} to compute $P(\Lambda)=\det\left( \frac{F}{\widetilde{F}}\widetilde{h}_{ij}+ \Lambda g_{ij}\right)$. As we are interested in computing the coefficient of $\Lambda^{\alpha}$, we must consider only the terms with $r=n-\alpha$, for $a= \left( i_1<\cdots < i_{\alpha} \right)$ and $b=\left( j_1<\cdots < j_{n-\alpha} \right)$. This gives the right hand side of equality $(a)$.   
 
While proving the formula for $\delta_{\alpha}$, we made use of the following arguments: 
 
 (b) The two sets $\{i_1,...,i_{\alpha}\}$ and $\{j_1,..., j_{n-\alpha}\}$ partition the set $\{1,2,...,n\}$ and therefore there is a permutation $\sigma_1 \in S_n$ such that $i_1=\sigma_1(1), ..., i_{\alpha}=\sigma_1(\alpha), j_1=\sigma_1(\alpha+1), ... , j_{n-\alpha}=\sigma_1(n)$. While $i_1<\cdots < i_{\alpha}$ and $j_1<\cdots < j_{n-\alpha}$ are in a fixed order, the values of the permutation $\sigma_1$ (grouped in two sets, one with $\alpha$ elements and one with $n-\alpha$ elements) are not ordered. Hence the factor $1/\alpha!(n-\alpha)!$ in front of the sum.
 
 (c) If we denote $\sigma_2=\sigma\sigma_1$, then, for a fixed $\sigma_1\in S_n$, $\sigma$ covers $S_n$ if and only if $\sigma_2$ covers $S_n$. Their signatures are related by $\varepsilon(\sigma) = \varepsilon(\sigma_1\sigma\sigma_1) =\varepsilon(\sigma_1\sigma_2)$.  
 
 (d) We use formula \eqref{gh} to replace the metric tensor $g_{ij}$ in terms of the angular metric $h_{ij}$, of rank $n-1$, and the components $F_{y^i}F_{y^j}$ of rank $1$. Therefore, in each product of $\alpha$ factors
\begin{eqnarray}
 \left(h_{\sigma_1(1)\sigma_2(1)} + \frac{\partial F}{\partial y^{\sigma_1(1)}} \frac{\partial F}{\partial y^{\sigma_2(1)}}\right) \cdots \left(h_{\sigma_1(\alpha)\sigma_2(\alpha)} + \frac{\partial F}{\partial y^{\sigma_1(\alpha)}} \frac{\partial F}{\partial y^{\sigma_2(\alpha)}}\right), \label{tehn1}
\end{eqnarray} 
 the components $\frac{\partial F}{\partial y^{\sigma_1(j)}} \frac{\partial F}{\partial y^{\sigma_2(j)}}$ will appear exactly once, for each $j\in \{1,...,\alpha\}$. 
 
The terms of formula \eqref{tehn1}, where the components of $F_{y^i}F_{y^j}$ do not appear, are the products $h_{\sigma_1(1)\sigma_2(1)}\cdots h_{\sigma_1(\alpha)\sigma_2(\alpha)}$. These products, in the right hand side of equality $(c)$, will come together with the products    
$\widetilde{h}_{\sigma_1(\alpha+1)\sigma_2(\alpha+1)}\cdots \widetilde{h}_{\sigma_1(n)\sigma_2(n)}$, and their corresponding sum
\begin{eqnarray*}
\sum_{\sigma_1, \sigma_2 \in S_n} \varepsilon(\sigma_1 \sigma_2) h_{\sigma_1(1)\sigma_2(1)} \cdots h_{\sigma_1(\alpha)\sigma_2(\alpha)} \widetilde{h}_{\sigma_1(\alpha+1)\sigma_2(\alpha+1)} \cdots \widetilde{h}_{\sigma_1(n)\sigma_2(n)}
\end{eqnarray*}
will vanish since it represents the coefficient of $\Lambda^{\alpha}$ of the polynomial $\det\left(\widetilde{h}_{ij}+\Lambda h_{ij}\right)=0$. Here I used that both angular metrics have the same kernel, which implies that $\left(\widetilde{h}_{ij}+\Lambda h_{ij}\right) y^j=0$ and hence $\det\left(\widetilde{h}_{ij}+\Lambda h_{ij}\right)=0$.
 
While expanding formula \eqref{tehn1}, there are $\alpha$ identical terms that contain $\frac{\partial F}{\partial y^{\sigma_1(j)}} \frac{\partial F}{\partial y^{\sigma_2(j)}}$, for each $j\in \{1,...,\alpha\}$. We group them and hence the factor in front of the sum from the right hand side of equality $(d)$ becomes $1/(\alpha-1)!(n-\alpha)!$. Then, we replace both permutations $\sigma_1$ and $\sigma_2$ with $\sigma_1 \circ \tau_{\alpha n}$ and $\sigma_2 \circ \tau_{\alpha n}$ ($\tau_{\alpha n}$ is the transposition that interchanges $\alpha$ and $n$) and obtain formula \eqref{deltak}.
   \end{proof}
 
The following lemma shows the importance of the coefficients $\delta_{\alpha}$ and $f_{\alpha}$, of the two characteristic polynomials $P$ and $Q$, for expressing the volume forms $\Omega_{\alpha}$ in a natural coordinate system or with respect to the canonical volume form.

\begin{lem} \label{omega0k}
Consider $F$ and $\widetilde{F}$ two projectively related Finsler
metrics. In terms of a natural coordinate system, the $(2n-1)$-forms
$\Omega_\alpha$ given by \eqref{omegak} can be
expressed as follows:
\begin{eqnarray}
\Omega_\alpha=(-1)^{n-1}\dfrac{\delta_{\alpha}}{F^{n}} dx\wedge
  i_{\mathcal{C}}dy, \quad \forall \alpha\in \{1,...,n\},  \label{okdxtau}
\end{eqnarray}
where $\delta_{\alpha}$ are the coefficients of the polynomial $P$  and 
$$ i_{\mathcal{C}}dy=\sum_{i=1}^n (-1)^{i-1}y^i dy^1 \wedge \cdots dy^{i-1}\wedge dy^{i+1} \wedge \cdots \wedge dy^n.$$
With respect to the volume form $\Omega_{F}$, we have:
\begin{eqnarray}
\Omega_\alpha=f_{\alpha} \Omega_{F}, \quad \forall \alpha \in \{1,...,n\},  \label{oosm}
\end{eqnarray}
where $f_{\alpha}$ are the coefficients of the polynomial $Q$.
\end{lem}

\begin{proof}
Since $S$ is the geodesic spray of the Finsler metric $F$, it follows that $d_hF=0$, and therefore we have:
\begin{eqnarray*}
\dfrac{1}{F}dF=\dfrac{1}{F}d_vF=\dfrac{1}{F}\dfrac{\partial F}{\partial y^i}\delta y^i. \label{dvf}
\end{eqnarray*}
Together with the $(2n-1)-$forms \eqref{omegak}, we will consider the following $2n-$forms, which are volume forms on $T_0M$ for each $\alpha \in \{1,..,n\}$:
\begin{eqnarray*}
\nonumber \Omega'_{\alpha}  & = & \dfrac{1}{F}dF \wedge \Omega_{\alpha} = \dfrac{(-1)^{(n-1)(n-2)/2}}{(\alpha-1)!(n-\alpha)!}\dfrac{1}{F}d_vF \wedge d_JF \wedge \left(dd_JF\right)^{(\alpha-1)} \wedge \left(dd_J\widetilde{F}\right)^{(n-\alpha)} \\
\nonumber & =  & \dfrac{(-1)^{(n-1)(n-2)/2}}{(\alpha-1)!(n-\alpha)!} \dfrac{1}{F} \dfrac{\partial F}{\partial y^i} \delta y^i \wedge \dfrac{\partial F}{\partial y^j} dx^j \wedge \left( \dfrac{1}{F} h_{ij} \delta y^i \wedge dx^j\right)^{(\alpha -1)} \wedge \left( \dfrac{1}{\widetilde{F}} \widetilde{h}_{ij} \delta y^i \wedge dx^j\right)^{(n-\alpha)} \\
& \stackrel{\eqref{deltak}}{=} & \dfrac{(-1)^{(n-1)(n-2)/2}}{F^n} \delta_{\alpha} dy^1 \wedge dx^1 \wedge \cdots \wedge dy^n \wedge dx^n = \dfrac{-\delta_{\alpha}}{F^n} dx\wedge dy.
\end{eqnarray*}
Now, we can prove formula \eqref{okdxtau}:
\begin{eqnarray*}
\Omega_{\alpha}=i_{\mathcal{C}}\Omega'_{\alpha}=\dfrac{(-1)^{n-1}\delta_{\alpha}}{F^n}dx\wedge i_{\mathcal{C}}dy. 
\end{eqnarray*}

In order to prove formula \eqref{oosm}, we pay attention now to the volume form $\Omega_n=\Omega_{F}$. The coefficient $\delta_n$ of the polynomial $P$ is given by $\delta_n=\det g$. Therefore, for $\alpha=n$, formula \eqref{okdxtau} reads 
\begin{eqnarray}
\Omega_n=(-1)^{n-1}\frac{\delta_n}{F^n}dx\wedge i_{\mathcal{C}}dy
  = (-1)^{n-1}\frac{\det g}{F^n}dx\wedge
  i_{\mathcal{C}}dy. \label{oodxtau}
\end{eqnarray}
Now, the two formulae \eqref{okdxtau} and \eqref{oodxtau} provide the formula \eqref{oosm} that expresses the $(2n-1)$-forms $\Omega_{\alpha}$ in terms of the volume form $\Omega_n$. 
  \end{proof}
  The volume form $\Omega_n=\Omega_{F}$ has been introduced in \cite{HS05} using an orthonormal frame. In \cite[(1.5)]{HS05}, the expression of this volume form is given in terms of a natural coordinate system and does not contain the factor $(-1)^{n-1}$, which appears in formula \eqref{oodxtau}.
  
  \section{Proof of Theorem \ref{mainthm} and some particular cases}

In this section we provide the proof of the main result, Theorem \ref{mainthm}, and use it to obtain explicit formulae for the first integrals $f_{\alpha}$, for $\alpha=1$ and $\alpha=n-1$. The expression for $f_1$ has been obtained before in \cite[Lemma 3.1]{BCC21}, using different techniques. The expression for $f_{n-1}$ is new. We show that in the Riemannian case, the two integrals \eqref{io} and \eqref{in1} reduce to the two first integrals $I_0$ and $I_1$ obtained by Topalov and Matveev in \cite[Theorem 1]{TM03}. 

In the previous section, in Lemma \ref{omega0k}, we have seen that all volume forms $\Omega_{\alpha}$, given by \eqref{omegak}, can be expressed in terms of the canonical volume form  $\Omega_F=\Omega_n$, the proportionality factors being given by formula \eqref{oosm}:
\begin{eqnarray*}
\Omega_{\alpha}=f_{\alpha}\Omega_F, \quad \forall \alpha \in \{1,..,n\}, 
\end{eqnarray*}  
for $f_{\alpha}$ the coefficients of the characteristic polynomial \eqref{cpq}.
In order to complete the proof of Theorem \ref{mainthm}, we will show now that the volume forms $\Omega_{\alpha}$ are invariant by the Reeb vector field.

\begin{lem} \label{invforms}
For two projectively related Finsler metrics $F$ and $\widetilde{F}$, the $(2n-1)$-forms \eqref{omegak} are invariant by the Reeb vector field $S/F$. 
\end{lem}
\begin{proof}
According to formula \eqref{omegak}, it suffices to prove that $\mathcal{L}_{S/F}d_JF=0$, $\mathcal{L}_{S/F}dd_JF=0$ and $\mathcal{L}_{S/F}dd_J\widetilde{F}=0$. 

First, we rewrite the Rapcsack equation $(R_1)$ in terms of the Reeb vector field $S/F$:
\begin{eqnarray}
\mathcal{L}_{S/F}d_J\widetilde{F}=\dfrac{1}{F}\mathcal{L}_Sd_J\widetilde{F} - \widetilde{F}\dfrac{dF}{F^2} = \dfrac{1}{F} \left(\mathcal{L}_Sd_J\widetilde{F} - d\widetilde{F} \right) +  d\left(\dfrac{\widetilde{F}}{F}\right) \stackrel{(R_1)}{=} d\left(\dfrac{\widetilde{F}}{F}\right). \label{r1'}
\end{eqnarray}
From the above formula \eqref{r1'}, we obtain
that the Hilbert forms $d_JF$, $dd_JF$ and $dd_J\widetilde{F}$ are invariant by the Reeb vector field $S/F$ and hence the volume forms \eqref{omegak} are invariant as well.
\end{proof}

Next, we use formulae \eqref{deltak} to obtain some explicit values for the first integrals $f_{\alpha}$. 

For $\alpha=1$, we have $f_1=\delta_1/\det g$. From formula \eqref{deltak}, for $\alpha=1$, we have 
 \begin{eqnarray} 
\nonumber \delta_{1}   & = & \dfrac{F^{n-1}}{\widetilde{F}^{n-1}}\frac{1}{(n-1)!}\sum_{\sigma_1, \sigma_2 \in S_n} \varepsilon(\sigma_1\sigma_2)  \prod_{i=1}^{n-1}\widetilde{h}_{\sigma_1(i)\sigma_2(i)} \frac{\partial
            F}{\partial y^{\sigma_1(n)}} \frac{\partial F}{\partial
                 y^{\sigma_2(n)}}  \\
\nonumber & = & \dfrac{F^{n-1}}{\widetilde{F}^{n-1}}\frac{1}{(n-1)!}\sum_{j=1}^n \sum_{\sigma_1, \sigma_2 \in S_n, \sigma_1(n)=j} \varepsilon(\sigma_1\sigma_2)  \prod_{k=1, k\neq j}^{n} \widetilde{h}_{k\sigma_2 \sigma_1^{-1}(k)}
\frac{\partial F}{\partial y^{j} } \frac{\partial F}{\partial y^{\sigma_2\sigma_1^{-1}(j)}}  \\
\nonumber & \stackrel{\sigma_3=\sigma_2 \sigma_1^{-1}}{=} & \dfrac{F^{n-1}}{\widetilde{F}^{n-1}}\frac{1}{(n-1)!}\sum_{j=1}^n \sum_{\sigma_1, \sigma_3 \in S_n,
      \sigma_1(n)=j} \varepsilon(\sigma_3)  \prod_{k=1, k\neq
      j}^{n}h_{k\sigma_3(k)} \frac{\partial F}{\partial y^{j} } \frac{\partial F}{\partial y^{\sigma_3(j)}}  \\
\nonumber  & = & \dfrac{F^{n-1}}{\widetilde{F}^{n-1}} \sum_{j=1}^n \sum_{\sigma_3 \in S_n} \varepsilon(\sigma_3)  \prod_{k=1, k\neq
      j}^{n} \widetilde{h}_{k\sigma_3(k)} \frac{\partial F}{\partial y^{j} } \frac{\partial F}{\partial y^{\sigma_3(j)}}  \\  
&  = & \dfrac{F^{n-1}}{\widetilde{F}^{n-1}}\det \left(\widetilde{h}_{ij}+ \frac{\partial
            F}{\partial y^i} \frac{\partial F}{\partial
            y^j}\right) = \frac{F^{n+1}}{\widetilde{F}^{n+1}}\det \widetilde{g}. \label{deltan1}
\end{eqnarray}
For the last equality above we made use of \cite[Lemma 2.2]{BCC21}. If we replace the value of $\delta_{1}$ in $f_1=\delta_1/\det g$, we obtain the first integral \eqref{io}.

If we denote: 
\begin{eqnarray*}
\mu = \left(\dfrac{\det g}{\det \widetilde{g}}\right)^{\frac{1}{n+1}},
\end{eqnarray*}
then formula \eqref{io} for the first integral $f_1$ can be written as:
\begin{eqnarray}
\dfrac{F^2}{f_1^{2/n+1}}=\mu^2 \widetilde{F}^2. \label{ri0}
\end{eqnarray}
Therefore, the right hand side of formula \eqref{ri0} is a first integral. In the Riemannian case it is quadratic in velocities and it reduces to the Painlev\'e first integral $I_0$, \cite[Remark 1]{TM03}. An equivalent expression of this first integral has been obtained by Crampin in \cite{Crampin03} as a $0^{+}$-homogeneous function,
$\kappa= \frac{F}{\mu \widetilde{F}}$.  The right hand side of formula \eqref{ri0} can be obtained also as a first integral for two projectively related Finsler metrics using \cite[Theorem 3]{Sarlet07}.

For $\alpha=n-1$, we obtain another first integral using directly the fact that $f_{\alpha}$ are the coefficients of the polynomial \eqref{cpq}:
\begin{eqnarray*}
f_{n-1}=\operatorname{Tr}\left(\mathcal{H}^i_j\right) = \dfrac{F}{\widetilde{F}}g^{ij}\left(\widetilde{g}_{ij} - \dfrac{\partial\widetilde{F}}{\partial y^i}  \dfrac{\partial \widetilde{F}}{\partial y^j} \right).
\end{eqnarray*} 
The above expression for the first integral $f_{n-1}$ can be written as:
\begin{eqnarray}
f_{n-1} \dfrac{\widetilde{F}^3\mu^3}{F} & = & \mu^3 g^{ij}\left(\widetilde{g}_{ij} \widetilde{g}_{kl} - \widetilde{g}_{ik} \widetilde{g}_{jl} \right)y^ky^l. \label{ri1}
\end{eqnarray}
It follows that the right hand side of formula \eqref{ri1} is a first integral. In the Riemannian case it is quadratic in velocities and corresponds to the the first integral $I_1$ from \cite[Theorem 1]{TM03}.

In dimension $2$, we can obtain directly the $0$-homogeneous integral $f_1$ and then show that it agrees with either one of the two formulae \eqref{io} or \eqref{in1}. Consider $F$ and $\widetilde{F}$ two geodesically equivalent Finsler metrics on a $2$-dimensional manifold. The forms \eqref{omegak} are now
\begin{eqnarray*}
\Omega_1=d_JF \wedge dd_J\widetilde{F}, \quad \Omega_2=d_JF \wedge dd_JF. \label{omega01}
\end{eqnarray*} 
For the volume form $\Omega_2$ we have:
\begin{eqnarray*}
\Omega_2 & = & \frac{\partial F}{\partial y^i} dx^i \wedge \frac{1}{F} h_{ij} \delta y^i \wedge dx^j = \frac{1}{F} \frac{\partial F}{\partial y^i} dx^i \wedge h_{ij}  dy^i \wedge dx^j \\ & \stackrel{(a)}{=} &  \frac{1}{F^2} \left( g_{2i}y^i h_{j1}dy^j - g_{1i}y^i h_{j2}dy^j \right) \wedge dx^1 \wedge dx^2 \\ & \stackrel{(b)}{=} & \frac{\det g}{F^2} \left(y^2 dy^1 - y^1 dy^2\right) \wedge dx^1 \wedge dx^2 = \frac{-\det g}{F^2}  dx^1 \wedge dx^2 \wedge \left(y^1 dy^2 - y^2 dy^1\right),  
\end{eqnarray*}
which is formula \eqref{oodxtau}, for $n=2$. In the previous calculations we used:

(a) $\frac{\partial F}{\partial y^i} = \frac{1}{F}g_{ij}y^j$;

(b) the expression \eqref{gh} of the angular metric $h_{ij}$ in terms of the metric $g_{ij}$.

 For the volume form $\Omega_1$ we have:
\begin{eqnarray*}
\Omega_1 & = & \frac{\partial F}{\partial y^i} dx^i \wedge \frac{1}{\widetilde{F}} \widetilde{h}_{ij} \delta y^i \wedge dx^j = \frac{1}{\widetilde{F}} \frac{\partial F}{\partial y^i} dx^i \wedge \widetilde{h}_{ij}  dy^i \wedge dx^j \\  & = &  \frac{1}{\widetilde{F}} \left( \frac{\partial F}{\partial y^2} \widetilde{h}_{11} - \frac{\partial F}{\partial y^1} \widetilde{h}_{12} \right) dx^1 \wedge dx^2 \wedge dy^1 + \frac{1}{\widetilde{F}} \left( \frac{\partial F}{\partial y^2} \widetilde{h}_{21} - \frac{\partial F}{\partial y^1} \widetilde{h}_{22} \right) dx^1 \wedge dx^2 \wedge dy^2 \\ & \stackrel{(c)}{=} & \dfrac{-\delta_1}{F^2}  dx^1 \wedge dx^2 \wedge \left(y^1 dy^2 - y^2 dy^1\right).  
\end{eqnarray*}
In the equality $(c)$ above, we introduce the function $\delta_1$ to show that for $n=2$ we recover also formula \eqref{okdxtau}. In order to obtain the function $\delta_1$, we identify the terms in both sides of equality (c). We obtain:
\begin{eqnarray*}
y^2 \dfrac{\delta_1}{F^2} & = & \frac{1}{\widetilde{F}} \left( \frac{\partial F}{\partial y^2} \widetilde{h}_{11} - \frac{\partial F}{\partial y^1} \widetilde{h}_{12} \right) \\
-y^1 \dfrac{\delta_1}{F^2} & = & \frac{1}{\widetilde{F}} \left( \frac{\partial F}{\partial y^2} \widetilde{h}_{21} - \frac{\partial F}{\partial y^1} \widetilde{h}_{22} \right).
\end{eqnarray*}
We multiply the first equation above by ${\partial F}/{\partial y^2}$, the second one by $-{\partial F}/{\partial y^1}$, we sum these expressions and obtain:
\begin{eqnarray*}
\delta_1 \dfrac{\widetilde{F}}{F} = - \begin{vmatrix}
\widetilde{h}_{11} & \widetilde{h}_{12} & \dfrac{\partial F}{\partial y^1} \\
\widetilde{h}_{21} & \widetilde{h}_{22} & \dfrac{\partial F}{\partial y^2} \\
\dfrac{\partial F}{\partial y^1} & \dfrac{\partial F}{\partial y^2} & 0
\end{vmatrix} \stackrel{(d)}{=} \dfrac{F^2}{\widetilde{F}^2} \det \widetilde{g}, \quad \delta_1= \dfrac{F^3}{\widetilde{F}^3} \det \widetilde{g}.
\end{eqnarray*}
For the equality $(d)$ above we did use \cite[Lemma 2.2]{BCC21}. The above expression for $\delta_1$ agrees, for $n=2$, with the expression \eqref{deltan1}.

With the value of the function $\delta_1$, we obtain 
\begin{eqnarray*}
\Omega_1 = \frac{-F}{\widetilde{F}^3} \det \widetilde{g} \ dx^1 \wedge dx^2 \wedge \left(y^1 dy^2 - y^2 dy^1\right).
\end{eqnarray*}
By comparing the expressions for the volume forms $\Omega_1$ and $\Omega_2$, we obtain
\begin{eqnarray*}
\Omega_{1}=f_1 \Omega_2, \quad f_1 = \frac{F^{3}}{\widetilde{F}^{3}}\frac{\det \widetilde{g}}{\det g},  
\end{eqnarray*}
which agrees with formula \eqref{io}. We will show that in the $2$-dimensional case, this formula agrees also with \eqref{in1}.
From \eqref{in1}, for $n=2$, we have
\begin{eqnarray*}
f_1 & = & \dfrac{F}{\widetilde{F}}\left(g^{11}\widetilde{h}_{11} + g^{12}\widetilde{h}_{12} + g^{21}\widetilde{h}_{21}+ g^{22}\widetilde{h}_{22} \right) \\ & = & \dfrac{F}{\widetilde{F}} \dfrac{1}{\det g}\left(g_{22}\widetilde{h}_{11} - g_{12}\widetilde{h}_{21} - g_{21}\widetilde{h}_{12}+ g_{11}\widetilde{h}_{22} \right) = \dfrac{\delta_1}{\det g} = \dfrac{F^3}{\widetilde{F}^3} \dfrac{\det \widetilde{g}}{\det g}.
\end{eqnarray*}

In the $2$-dimensional case, the proportionality factor and hence the first integral $f_1$,  can be obtained using a different argument. The angular metrics  $\widetilde{h}_{ij}$ and $h_{ij}$ have rank $1$ and they are proportional, $\frac{1}{\widetilde{F}}\widetilde{h}_{ij}=f_1 \frac{1}{F}h_{ij}$. It follows the proportionality of the Hilbert $2$-forms $ dd_J \widetilde{F} = f_1 dd_JF$ and hence the proportionality of the induced volume forms on $SM$, $\Omega_{1}=f_1 \Omega_2$.   

\subsection*{Acknowledgements} I am grateful to Vladimir Matveev for very useful discussions regarding the Riemannian correspondents of the first integrals \eqref{io} and \eqref{in1}. I express my thanks to the referees for their comments and suggestions that improved the proofs of the results.

\end{document}